\documentclass{amsart}
\usepackage{graphicx}
\usepackage{amssymb}
\newtheorem{thm}{Theorem}[section]

\newtheorem{lemma}[thm]{Lemma}

\theoremstyle{definition}

\theoremstyle{question}

\theoremstyle{Conjecture}
\newtheorem{con}[thm]{Conjecture}

\numberwithin{equation}{section}

\begin{document}

\title[A counterexample to Herzog's Conjecture on the number of involutions]{A counterexample to Herzog's Conjecture on the number of involutions}%
\author{M. Zarrin}%

\address{Department of Mathematics, University of Kurdistan, P.O. Box: 416, Sanandaj, Iran}%
 \email{M.zarrin@uok.ac.ir}
\begin{abstract}
 In 1979, Herzog put forward the following conjecture: if two simple groups have the same
number of involutions, then they are of the same order. We give a counterexample to this conjecture.
\end{abstract}
\maketitle

\section{\textbf{ Introduction}}

Let $G$ be a group, $\pi(G)$ be the set of primes $p$ such that $G$ contains an element of order $p$  and $I_k(G)$ be the number of elements of order $k$ in $G$ (in some other papers denoted by $s_k(G)$ or $m_k(G)$).

M. Herzog in 1979 \cite{Herzog}, showed that there are vast classes of simple groups in which characterized by
the number of its involutions.

\begin{thm}
Let $G$ be a finite simple group with $I$ involutions and suppose that
$I\equiv 1 ~(mod~ 4)$. Then one of the following holds:
\begin{itemize}
\item $I = 1$ and $G$ is cyclic of order $2$,
\item $I= 105$ and $G\cong A_7$,
\item $I= 165$ and $G\cong M_{11}$,
\item $I= q(q + \varepsilon)/2$, and $G\cong PSL(2, q)$, where $q = p^n > 3$ is a power of an
odd prime, $\varepsilon =1$ or $-1 $ and  $q\equiv \varepsilon ~ (mod~ 8)$,
\item $I= q^2(q^2 + q + 1)$ and $G =PSL(3, q)$, where $q = p^n$ is a power of an
odd prime and $q\equiv -1 ~(mod ~4)$,
\item $I = q^2(q^2-q + 1)$ and $G =± PSU(3, q)$, where $q = p^n$ is a power of an "
odd prime and $q\equiv 1 ~(mod ~ 4)$.
\end{itemize}
\end{thm}
In fact, He showed that each of the above mentioned simple groups is characterized by
the number of its involutions. Then in view of the above results and as the groups $A_8$ and $PSL(3, 4)$ (it is well-known $G=PSL(3,4)$ has a single class of involutions, and its centralizer is the Sylow 2-subgroup, say $P$. It follows that $I_2(G)=|G|/|P|=315$) are of the same order and
each has $315$ involutions, he gave the following conjecture:
 \begin{con}
 If two simple groups have the same
number of involutions, then they are of the same order.
\end{con}
Here we provide a counterexample to
this conjecture and give related questions. Let $p$ denote an odd prime integer and let $q =p^n$
where $n$ is a positive integer. Let $PSp(4, q)$ denote the projective symplectic group
in dimension $4$ over a field $F_q$ of $q$ elements. We show that the number of involution of the projective symplectic group of degree of degree $4$ over the finite field of size $3$, $PSp(4,3)$ is $315$.  In fact,  $$I_2(PSp(4,3))=315=I_2(PSL(3,4)),$$ but $|PSp(4,3)|=25920$ and $|PSL(3,4)|=20160$. \\

The group $G=PSp (4, q)$ is described in \cite{Wong}. This group is simple of order $\frac{1}{2}q^4(q^2+1)(q^2-1)^2$ and has a Sylow $2$-subgroup with
center of order $2$ so that involutions which lie in the centers of Sylow $2$-subgroups
form a single conjugacy class. In particular, by Lemmas 1.1, 2.4 and 3.7 of \cite{Wong}, the group $PSp (4, 3)$ has exactly two classes of involutions, say $t$ and $u$ and  $|C_{PSp (4, 3)}(t)|=3^2(3^2-1)^2=576$ and $|C_{PSp (4, 3)}(u)|=12\times 8=96$ (note that $t$ lie in the centers of Sylow $2$-subgroups).
Now by the following lemma, we can obtain that $I_2(PSp (4, 3))=\frac{25920}{576}+\frac{25920}{96}=315$.

\begin{lemma}
Let $H$ be a finite group with $k_2$ distinct conjugacy classes of involutions, and if $t_1, t_2, \ldots, t_{k_2}$ are involution elements of $H$, one form each of these $k_2$ classes, then $I_2(H)=\sum_{i=1}^{k_2} |H:C_H(t_i)|$, where $|H:C_H(t_i)|$ is the index of the centralizer $t_i$ in $H$.
\end{lemma}
\begin{proof}
Clearly.
\end{proof}
In view of his results and comparison with other results related to the set  $\{I_t(G)\mid t\in \pi(G)\}$ (see for instance, \cite{Zar}), one can find out the influence of $I_2(G)$ is more stronger than other $I_p(G)$ where $p$ is a prime number. The authors in \cite{Zar} conjectured that: if $G$ is a finite nonabelian simple group. Then $I_p(G)\neq I_q(G)$ for all
distinct prime divisors $p$ and $q$ of $|G|$. Finally, in view of Herzog's Conjecture and Conjecture 2.10 of \cite{Zar}, it might seem reasonable to make the following conjecture.

\begin{con}
If $S$ be a non-abelian simple and $G$ a group such
that $I_2(G)=I_2(S)$ and $I_p(G)=I_p(S)$ for some odd prime divisor $p$. Then $|G|=|S|$.
\end{con}

\end{document}